\documentclass[9pt]{extarticle} % for 9pt font
\usepackage[ % for odd page size and margins
  paperwidth=5.88in,paperheight=8.60in,
  textwidth=4.2in,textheight=6.7in]{geometry} % 860px x 588px image, 389px wide textblock

\usepackage{fancyhdr} % define headers/footers
\usepackage{amsmath, amsthm, amsfonts, amssymb, mathdots}
\usepackage{stmaryrd}
\usepackage{paralist}
\usepackage{mathtools}
\usepackage{tikz}
\usepackage{graphicx}
\usepackage{cleveref}
\newtheorem{theorem}{Theorem}

\newtheorem{lemma}{Lemma}

\newtheorem{proposition}{Proposition}
\theoremstyle{definition}
\newtheorem{remark}{Remark}

\fancypagestyle{headings}{%
\fancyhf{} % clear all header and footer fields
\fancyfoot[L]{%
}

}

\newcommand{\RR}{\mathbb{R}}
\newcommand{\CC}{\mathbb{C}}

\newcommand{\MM}{\textnormal{\textbf{M}}}
\newcommand{\Gl}{\textnormal{\textbf{Gl}}}

\newcommand{\fourdiag}{\hspace{0.02cm}\begin{tikzpicture}[scale=0.15,baseline=+1.5mm] \draw[color=black] (0,1) -- (1,0); \draw[color=black] (0,0.5) -- (0.5,0); \draw[color=black] (0.5,1) -- (1,0.5); \end{tikzpicture}}

\makeatletter
\renewcommand{\maketitle}{
\pagestyle{plain}
\vspace*{\baselineskip}
\begin{center}
\MakeUppercase{\small Ralph John de la Cruz} \\ \emph{Institute of Mathematics, University of the Philippines} \\ \emph{Diliman, Quezon City 1105, Philippines} \\ (E-Mail: rjdelacruz@math.upd.edu.ph)\\[0.25cm]
\MakeUppercase{\small Philip Saltenberger} \\ \emph{Institute for Numerical Analysis, TU Braunschweig} \\ \emph{Universit\"atsplatz 2, 38106 Braunschweig, Germany} \\ (E-Mail: philip.saltenberger@tu-bs.de) \\ Corresponding author

\vspace*{3\baselineskip}

\MakeUppercase{\large \@title}
\end{center} \vskip-\baselineskip
}
\makeatother

\usepackage{titlesec} % define section format
\titleformat{\section}[hang]{\scshape}{\thesection. }{0pt}{\centering}[]

\title{Generic canonical forms for perplectic and symplectic normal matrices}

\begin{document}
\maketitle
\begin{abstract}
Let $B = J_{2n}$ or $B=R_n$ for the matrices given by
$$J_{2n} = \left[ \begin{smallmatrix} & I_n \\ - I_n & \end{smallmatrix} \right] \in \MM_{2n}(\CC) \quad \textnormal{or} \quad R_{n} = \left[ \begin{smallmatrix}& & 1 \\ & \iddots & \\ 1 & &  \end{smallmatrix} \right] \in \MM_n(\CC).$$
A matrix $A$ is called $B$-normal if $AA^\star = A^\star A$ holds for $A$ and its adjoint matrix $A^\star := B^{-1}A^HB$. In addition, a matrix $Q$ is called $B$-unitary, if $Q^HBQ = B$. We develop sparse canonical forms for nondefective (i.e. diagonalizable) $J_{2n}/R_n$-normal matrices under $J_{2n}/R_n$-unitary similarity transformations. For both cases we show that these forms exist for an open and dense subset of $J_{2n}/R_n$-normal matrices. This implies that these forms can be seen as topologically 'generic' for $J_{2n}/R_n$-normal matrices since they exist for all such matrices except a nowhere dense subset.
\end{abstract}

\section{Introduction}
\label{s:intro}

For any arbitrary, nonsingular matrix $B \in \Gl_n(\CC)$ the function
$$[  \cdot , \cdot]: \CC^n \times \CC^n \rightarrow \CC, \; (x,y) \mapsto x^HBy \quad (x^H := \overline{x}^T) $$
defines a sesquilinear form on $\CC^n \times \CC^n$. If $B$ is Hermitian and positive definite, $[ \cdot, \cdot]$ is a scalar product; otherwise $[\cdot, \cdot]$ is often called an indefinite inner/scalar product \cite{Bog74,Mack07}.
Four well-known classes of matrices are related to indefinite scalar products, see also \cite{Mack07}:
\begin{enumerate}[(a)]
\item A matrix $A \in \MM_{n}(\CC)$ is called $B$-selfadjoint if $[Ax,y]=[x,Ay]$ holds for all $x,y \in \CC^{n}$. In other words,
$$x^HA^HBy = x^HBAy$$ is true for $A$ and all $x,y \in \CC^{n}$. This is possible if and only if $A^HB = BA$ holds, that is $A = B^{-1}A^HB$.
\item A matrix $A \in \MM_{n}(\CC)$ is called $B$-skewadjoint if $[Ax,y]=[x,-Ay]$ holds for all $x,y \in \CC^{n}$. It follows analogously to (a) that $A \in \MM_{n}(\CC)$ is skewadjoint if and only if $-A = B^{-1}A^HB$ holds.
\item A matrix $A \in \MM_{n}(\CC)$ is called $B$-unitary if $[Ax,Ay]=[x,y]$ holds for all $x,y \in \CC^{n}$. That is, $x^HA^HBy = x^HBy$ has to hold for $A$ and all vectors $x,y \in \CC^{n}$.
This is possible if and only if $A^HBA = B$.
\item A matrix $A \in \MM_{n}(\CC)$ is called $B$-normal if it holds that $$AB^{-1}A^HB = B^{-1}A^HBA.$$
\end{enumerate}

Indefinite scalar products arise in many mathematical contexts, see for instance \cite{GoLaRod05} for a comprehensive treatment of the Hermitian case $B=B^H$ with a selection of examples and applications.
Two particular choices for $B$ that are very frequently considered are
$$ R_n = \begin{bmatrix} & & 1 \\ & \iddots & \\ 1 & & \end{bmatrix} \in \Gl_n(\CC) \quad \textnormal{and} \quad  J_{2m} = \begin{bmatrix} & I_m \\ - I_m & \end{bmatrix} \in \Gl_{2m}(\CC).$$
Both define indefinite scalar products on $\CC^n \times \CC^n$ and $\CC^{2m} \times \CC^{2m}$, respectively. We call $[ \cdot, \cdot]_{R_n}: (x,y) \mapsto [x,y]_{R_n} = x^HR_ny$ the perplectic and $[ \cdot, \cdot]_{J_{2m}}: (x,y) \mapsto [x,y]_{J_{2m}} = x^HJ_{2m}y$ the symplectic scalar product.
The classes of matrices introduced in (a), (b), (c) and (d) above for the symplectic scalar product are of great importance in control systems theory, algebraic Riccati equations, gyroscopic systems, model reduction, quadratic eigenvalue problems and many more areas, see e.g. \cite{Francis, Lancaster, Mehr91, TissMeer} and the references therein. For the perplectic scalar product, such matrices arise for instance in control of mechanical and electrical vibrations, see e.g. \cite{Mack05,Reid}. In this work, our investigations focus entirely on $B=R_n$ and $B=J_{2m}$.

Notice that the set of $B$-normal matrices includes the sets of $B$-self\-ad\-joint, $B$-skew\-ad\-joint and $B$-unitary matrices. Now assume for a moment that $B = I_n$ is the $n \times n$ identity matrix. Then the sets of $B$-selfadjoint, $B$-skewadjoint, $B$-unitary and $B$-normal matrices $A \in \MM_n(\CC)$ coincide with the sets of Hermitian ($A=A^H$), skew-Hermitian ($A=-A^H$), unitary ($A^HA=I_n$) and normal ($AA^H=A^HA$) matrices. It is well-known that a matrix belonging to any of these sets is unitarily (i.e. $I_n$-unitarily) diagonalizable \cite{Grone}. Thus, their natural canonical form under unitary similarity is the diagonal form. If $B \neq I_n$, a  $B$-normal matrix is in general not diagonalizable by a $B$-unitary similarity (see \cite[Sec.\,9]{PS} and the specific conditions developed therein for $B$-unitary diagonalization). With this in mind, the question of a canonical form for $B$-normal matrices ($B=R_n, J_{2m}$) under $B$-unitary similarity naturally arises.
The following quote is taken from \cite{Mehl}\footnote{In our case, a '$H$-normal matrix' is a '$B$-normal matrix'.} and highlights the difficulties in finding such a canonical form:

\begin{quote}
\emph{``On the other hand, the problem of finding a canonical form for H-normal matrices
has been proven to be as difficult as classifying pairs of commuting matrices under
simultaneous similarity [...]. So far, a classification of H-normal matrices has only
been obtained for some special cases [...].
From this point of view, the set of all H-normal matrices is 'too large' and it
makes sense to look for proper subsets for which a complete classification can be
obtained.''}
\end{quote}

In this work we focus on the two popular cases $B=R_n$ and $B=J_{2m}$ and introduce a canonical form for a proper subset (as suggested in the quote above) of $B$-normal matrices under $B$-unitary similarity.  In particular, we only consider $B$-normal matrices which are nondefective (i.e. diagonalizable). Structured canonical (i.e. Schur, Jordan) forms of matrices belonging to the matrix-classes defined in (a), (b), (c) and (d) above were studied before, see e.g. \cite{ GoLaRod05, Laub, LinMehr, MehrXu, HamSchur} or \cite[Sec.\,7]{Mack05}. In full generality, they can be very complicated or their existence depends on specific properties (e.g. related to the eigenvalues) of the matrix at hand. The existence of the form we develop only depends on whether the matrix at hand is defective or not. Once a $B$-normal matrix does possess a full set of linearly independent eigenvectors (forming a basis of the whole vector space), it is possible to transform it into a sparse and nicely structured form via a $B$-unitary similarity.

Certainly, establishing a canonical form for a subset of $B$-normal matrices is only worth half its value as long as it is unknown how large this subset (for which the form exists) actually is.
The reason why we consider only diagonalizable matrices is, on the one hand, that these form a dense subset in the class of all $B$-normal matrices (and, in addition, they constitute dense subsets among all $B$-selfadjoint, $B$-skewadjoint and $B$-unitary matrices), see \cite{CruzSalten}. Therefore, any $B$-normal matrix is, if it is not diagonalizable, arbitrarily close to a diagonalizable $B$-normal matrix (for which the canonical form does exist).
On the other hand, once we have established the existence of the canonical form for diagonalizable matrices, we will use it to strengthen the results obtained in \cite{CruzSalten}. In fact we are able to show that the set of matrices with pairwise distinct eigenvalues is a topologically open and dense subset of all $B$-normal matrices. This means that its complement is ``nowhere dense'' and therefore, from a topological point of view, is rather ``small''. For this reason, we call these canonical forms ``generic'' for $B$-normal matrices.

In Section \ref{s:perplectic} we show that, assuming diagonalizability, $R_n$-normal matrices can always be transformed into '$X$-form' by a $R_n$-unitary similarity (see Theorem \ref{thm:normal2X}). We say that a matrix $A \in \MM_n(\CC)$ has $X$-form, if it has entries only along its diagonal and anti-diagonal:
$$ A = \left[ \tikz[baseline=+4mm]{ \draw[color=black] (0,1) -- (1,0); \draw[color=black] (0,0) -- (1,1); } \right].$$
We prove in Section \ref{ss:genericity} that the subset of $R_n$-normal matrices, for which such a canonical form exists, is open and dense.

The canonical form under $J_{2m}$-unitary similarity we develop for nondefective $J_{2m}$-normal matrices in Section \ref{s:symplectic} is the 'four-diagonal-form'. We say that $A \in \MM_{2m}(\CC)$ has four-diagonal-form if, considered as a $2 \times 2$ matrix with four $m \times m$ blocks, each of these blocks is diagonal:
\begin{equation} A = \left[ \tikz[baseline=+4mm]{ \draw[color=black] (0,1) -- (1,0); \draw[color=black] (0,0.5) -- (0.5,0); \draw[color=black] (0.5,1) -- (1,0.5); } \right]. \notag \end{equation}
Here, we also prove that this form can be seen as generic for the set of $J_{2m}$-normal matrices. Our proofs in this section will make essential use of the results obtained in Section \ref{s:perplectic}. Some general basics on indefinite scalar products and those matrices related to them are presented in Section \ref{s:general-basics}. Section \ref{s:conc} presents some conclusions.

\section{Basic definitions and notation}
\label{s:general-basics}

The set of all matrices of size $j \times k$  over $\mathbb{F} = \RR, \CC$ is denoted by $\MM_{j \times k}(\mathbb{F})$. For $j=k$, we use the notation $\MM_k(\mathbb{F}) = \MM_{k \times k}(\mathbb{F})$. The general linear group over $\mathbb{F}^k$  (that is, the set of $k \times k$ nonsingular matrices over $\mathbb{F}$) is denoted by $\Gl_k(\mathbb{F})$. The identity matrix of size $k \times k$ is denoted by $I_k$. Whenever $A \in \MM_k(\CC)$, the set of all eigenvalues of $A$ is called the spectrum of $A$ and is denoted $\sigma(A)$. The multiplicity of $\lambda \in \CC$ as a root of $\det(A - x I_k)$ is called the algebraic multiplicity of $\lambda$ (as an eigenvalue of $A$) and is denoted by $\mathfrak{m}(A, \lambda)$. Any vector $v \in \CC^k$ that satisfies $Av = \lambda v$ is called an eigenvector for $\lambda$. The set of all those vectors corresponding to $\lambda \in \sigma(A)$ is called the eigenspace for $\lambda$. The eigenspace of $A$ for $\lambda$ is a subspace of $\CC^k$. Its dimension is called the geometric multiplicity of $\lambda$. The conjugate transpose of a vector/matrix is denoted with the superscript $^H$ while $^T$ is used for the transposition without complex conjugation. A matrix $A \in \MM_n(\CC)$ is called nondefective (or diagonalizable), if there exists some $S \in \Gl_n(\CC)$ such that $S^{-1}AS$ is a diagonal matrix.

In this section, let $B = \pm B^H \in \Gl_n(\CC)$ be some nonsingular matrix and $[ \cdot, \cdot ]_B$ the indefinite scalar product induced by $B$ on $\CC^n \times \CC^n$. In this section we collect some basic results on $B$-selfadjoint, $B$-skewadjoint, $B$-unitary and $B$-normal matrices (recall that those matrices have been defined in Section \ref{s:intro}) and relations for their eigenvalues and eigenvectors. The following Lemma \ref{lem:scalarprod1} states a central property of eigenvectors of $B$-self/skewadjoint matrices related to the scalar product $[ \cdot, \cdot ]_B$.

\begin{lemma} \label{lem:scalarprod1} Let $B= \pm B^H \in \Gl_n(\CC)$ and $A \in \MM_{n}(\CC)$.
\begin{enumerate}[(a)]
\item Suppose $A$ is $B$-selfadjoint and $x,y \in \CC^{n}$ are eigenvectors of $A$ for $\lambda \in \CC$ and $\mu \in \CC$, respectively. Then $x^HBy \neq 0$ implies $\mu = \overline{\lambda}$.
\item Suppose $A$ is $B$-skewadjoint and $x,y \in \CC^{n}$ are eigenvectors of $A$ for $\lambda \in \CC$ and $\mu \in \CC$, respectively. Then $x^HBy \neq 0$ implies $\mu = -\overline{\lambda}$.
\end{enumerate}
\end{lemma}

\begin{proof}
(a) Under the given assumptions we have
$$ \overline{\lambda} \cdot [x,y]_B = [\lambda x,y]_B = [Ax,y]_B = [x,Ay]_B = [x,\mu y]_B = \mu \cdot [x,y]_B$$
thus $\mu = \overline{\lambda}$ has to hold whenever $[x,y]_B \neq 0$. The proof for (b) proceeds analogously.
\end{proof}

For any $A \in \MM_n(\CC)$ let $A^\star := B^{-1}A^HB$. The matrix $A^\star$ is usually referred to as the adjoint for $A$ \cite[Sec.\,2]{Mack07}. The sets of all $B$-selfadjoint, $B$-skewadjoint, $B$-unitary and $B$-normal matrices can now equivalently be characterized by the equations $A=A^\star$, $A=-A^\star$, $A^\star = A^{-1}$ and $AA^\star = A^\star A$, respectively. In particular, notice that a $B$-unitary matrix is always nonsingular. For a $B$-selfadjoint matrix $A = A^\star$ we have
\begin{equation} \sigma(A) = \sigma(A^\star) =  \sigma(B^{-1}A^HB)=  \sigma(A^H) =  \overline{\sigma(A)} \label{equ:spectrum}
\end{equation}
so $\sigma(A)$ consists of tuples $(\lambda, \overline{\lambda})$ with $\lambda$ and $\overline{\lambda}$ having the same algebraic multiplicities. Analogously we obtain the eigenvalue pairings $(\lambda, -\overline{\lambda})$ and $(\lambda, 1/ \overline{\lambda})$ for $B$-skewadjoint and $B$-unitary matrices, respectively \cite{Mack06, Mack07}. Notice that, given two eigenvectors $x,y \in \CC^{2n}$ of some $B$-selfadjoint matrix $A$ for the eigenvalue $\lambda$, Lemma \ref{lem:scalarprod1} guarantees that $[x,y] \neq 0$ can only hold if $\lambda = \overline{\lambda}$. Therefore $[x,y] \neq 0$ implies $\lambda \in \RR$.

The proof of Theorem \ref{thm:sylvester-inertia} can be found in \cite[Sec.\,4.5]{HoJo}. The proof for (b) follows in the same way by noting that $iA$ is skew-Hermitian whenever $A$ is Hermitian and vice versa.

\begin{theorem} \label{thm:sylvester-inertia}
\begin{enumerate}
\item Let $A = A^H \in \Gl_n(\CC)$. Let $m_{+} \geq 0$ and $m_{-} \geq 0$ be the numbers of positive and negative eigenvalues of $A$, respectively. Then there exists some $Q \in \Gl_n(\CC)$ such that
$$ Q^HAQ = \begin{bmatrix} +I_{m_{+}} & \\ & -I_{m_{-}} \end{bmatrix}.$$
\item Let $A = -A^H \in \Gl_n(\CC)$. Let $m_{+} \geq 0$ and $m_{-} \geq 0$ be the numbers of positive and negative imaginary eigenvalues of $A$, respectively. Then there exists some $Q \in \Gl_n(\CC)$ such that
$$ Q^HAQ = \begin{bmatrix} +iI_{m_{+}} & \\ & -iI_{m_{-}} \end{bmatrix}.$$
\end{enumerate}
\end{theorem}

The following two well-known results will also be used frequently in the next section.
We call two matrices $A_1,A_2 \in \MM_n(\CC)$ simultaneously diagonalizable, if there exists a single matrix $S \in \Gl_n(\CC)$ such that $S^{-1}A_jS$ is a diagonal matrix for $j=1, 2$. The following Lemma \ref{lem:simult-diag}, see \cite[Thm.\,1.3.21]{HoJo}, relates simultaneous diagonalization to commutativity of matrices. Lemma 3 below, see \cite[Thm.\,1.3.10]{HoJo}, makes a statement about the diagonalizability of block-diagonal matrices.

\begin{lemma} \label{lem:simult-diag}
Assume $A_1,A_2 \in \MM_n(\CC)$ are both diagonalizable. Then $A_1$ and $A_2$ commute if and only if they are simultaneously diagonalizable.
\end{lemma}

\begin{lemma} \label{lem:block-diag}
Let $A_1 \in \MM_{m_1}(\CC), \ldots , A_k \in \MM_{m_k}(\CC)$ be some matrices and set $m :=m_1 + \cdots + m_k$. Define
$$ A = \begin{bmatrix} A_1 & & \\ & \ddots & \\ & & A_k \end{bmatrix}  \in \MM_m(\CC).$$
Then $A$ is diagonalizable if and only if all $A_j, j=1, \ldots , k$, are diagonalizable.
\end{lemma}

The following result can easily be verified by a straight forward calculation (see also \cite[Sec.\,1]{Mehl} or \cite[Lem.\,2]{CruzSalten}). In particular it shows that the four classes of matrices related to an indefinite scalar product $[x,y]_B = x^HBy$ are preserved under $B$-unitary similarity ($B=\pm B^H$ is not required for the result to hold). The proof is omitted.

\begin{lemma} \label{lem:switch2}
Let $B \in \Gl_n(\CC)$ and $A \in \MM_n(\CC)$. Furthermore, let $T \in \Gl_n(\CC)$ and consider
$$ A' := T^{-1}AT \quad \textnormal{and} \quad B' := T^HBT.$$
Then $A$ is $B$-selfadjoint/$B$-skewadjoint/$B$-unitary/$B$-normal if and only if $A' = T^{-1}AT$ is $B'$-selfadjoint/$B'$-skew-adjoint/$B'$-unitary/$B'$-normal.
\end{lemma}

\subsection{Basic properties of matrices related to $[ \cdot, \cdot]_{R_n}$}
\label{ss:basics-perp}

For any $n \in \mathbb{N}$ let the Hermitian (i.e. real symmetric) matrix $R_n$ be defined as
$$ R_n = \begin{bmatrix} & & 1 \\ & \iddots & \\ 1 & & \end{bmatrix} \in \Gl_n(\RR).$$
The eigenvalues of $R_k$ are $+1$ and $-1$ with multiplicities $\mathfrak{m}(R_n,+1)=\lceil n/2 \rceil$  and $\mathfrak{m}(R_n,-1)=\lfloor n/2 \rfloor$. It is common to refer to $R_n$-self/skewadjoint matrices as per/perskew-Hermitian, respectively. Matrices that are $R_n$-unitary are in general called perplectic \cite{Mack05}. Whenever we are considering per/perskew-Hermitian and perplectic matrices of different sizes, the value of $n$ will always be clear from the context, i.e. from the size of the matrix at hand.
In this section we introduce some notational simplifications to keep the terminology in Section \ref{s:perplectic} as compact as possible

For two matrices $P \in \MM_{2\ell}(\CC)$ and $Q \in \MM_k(\CC)$ with $P$ block-partitioned as a $2 \times 2$ matrix with $\ell \times \ell$ blocks, i.e.
\begin{equation} P = \begin{bmatrix} P_{11} & P_{12} \\ P_{21} & P_{22} \end{bmatrix} \in \MM_{2\ell}(\CC), \quad P_{ij} \in \MM_\ell(\CC), \label{equ:partition-P} \end{equation}
we define the perplectic sum $P \boxdot Q$ of $P$ and $Q$ as
\begin{equation} P \boxdot Q = \begin{bmatrix} P_{11} & & P_{12} \\ & Q & \\ P_{21} & & P_{22} \end{bmatrix} \in \MM_{n}(\CC), \quad (2\ell + k=n). \label{equ:perpsum-def} \end{equation}
Note that $P \boxdot Q$ is only defined if the size of $P$ is even\footnote{Whenever we use the perplectic sum $P \boxdot Q$ without further notice, the size of $P$ will always be even and clear from the context.}. Analogously to the Kronecker product of matrices, there are some properties of the perplectic sum that follow immediately from its definition \eqref{equ:perpsum-def}. We use the notation $\oplus$ to denote the direct sum of two matrices, i.e. $P \oplus Q = \textnormal{diag}(P,Q)$. The proof of Remark \ref{rem:perp_properties} is omitted.

\begin{remark} \label{rem:perp_properties}
Let $k, \ell \in \mathbb{N}_0$ and $n \in \mathbb{N}$ such that $2 \ell + k = n$.
\begin{enumerate}[(a)]
\item For $P \in \MM_{2\ell}(\CC)$ as in \eqref{equ:partition-P} and $Q \in \MM_k(\CC)$ it holds that $(P \boxdot Q)^H = P^H \boxdot Q^H$. Moreover, $(P \boxdot Q)^{-1} = P^{-1} \boxdot Q^{-1}$ if $P^{-1}$ and $Q^{-1}$ exist.
\item If $P,S \in \MM_{2\ell}(\CC)$ (both interpreted as $2 \times 2$ matrices with $\ell \times \ell$ blocks) and $Q,W \in \MM_k(\CC)$ then
\begin{equation} \big( P \boxdot Q \big) \big( S \boxdot W \big) = PS \boxdot QW. \label{equ:perpsum_1} \end{equation}
In particular, $P \boxdot Q$ and $S \boxdot W$ commute if and only if $PS = SP$ and $QW = WQ$ hold.
\item  For $P \in \MM_{2\ell}(\CC)$ as in \eqref{equ:partition-P} and $Q \in \MM_k(\CC)$, the perplectic sum $P \boxdot Q \in \MM_n(\CC)$ is per(skew)-Hermitian if and only if $P$ and $Q$ are both per(skew)-Hermitian (with respect to $R_{2\ell}$ and $R_k$, respectively).
\item For $P \in \MM_{2\ell}(\CC)$, $Q \in \MM_k(\CC)$ and their perplectic sum $P \boxdot Q \in \MM_n(\CC)$, there exists a permutation matrix $R \in \Gl_n(\CC)$ so that $R^{-1}(P \boxdot Q) R = P \oplus Q$.
\end{enumerate}
\end{remark}

As the Kronecker product, the perplectic sum $\boxdot$ is not commutative. For abbreviation, we sometimes use the notation
\begin{equation}  \boxdot_{i=1}^k P_i = P_1 \boxdot P_2 \boxdot \cdots \boxdot P_k := ( P_1 \boxdot P_2) \boxdot P_3) \boxdot \cdots ) \boxdot P_k \label{equ:perpsum-notation} \end{equation}
The following Lemma \ref{lem:perplectic} summarizes some properties of perplectic matrices with respect to the operations $\oplus$ and $\boxdot$. Here and in the following we use the notation $S^{-\star}$ to denote $(S^{-1})^\star = (S^\star)^{-1}$.

\begin{lemma} \label{lem:perplectic}
Let $k, \ell \in \mathbb{N}_0$ and $n \in \mathbb{N}$ such that $2 \ell + k = n$.
\begin{enumerate}[(a)]
\item  For any matrix $S \in \Gl_\ell(\CC)$ and any perplectic matrix $Q \in \Gl_k(\CC)$, the matrix $P := ( S \oplus S^{-\star} ) \boxdot Q \in \Gl_n(\CC)$ is perplectic. In particular, $P = (I_\ell \oplus I_\ell) \boxdot Q$ and, if $n = 2\ell$, $P = S \oplus S^{-\star}$ are perplectic.
\item For any two perplectic matrices $S \in \Gl_{2\ell}(\CC)$ and $Q \in \Gl_k(\CC)$, the matrix $P := S \boxdot Q \in \Gl_{n}(\CC)$ is perplectic.
\item For any number $j \in \mathbb{N}$ of perplectic matrices $P_1, P_2, \ldots , P_j \in \Gl_n(\CC)$ their product $P := P_1P_2 \cdots P_j$ is perplectic.
\end{enumerate}
\end{lemma}

\begin{proof}
All statements follow from straight forward calculations using Remark \ref{rem:perp_properties} (a) and (b):

(a) Noting that $S^{-\star} = R_\ell S^{-H} R_\ell$ we obtain
\begin{align*}
P^HR_nP &= \bigg( \big( S^H \oplus (S^{-\star})^H \big) \boxdot Q^H \bigg) \big( R_{2\ell} \boxdot R_k \big) \bigg( \big( S \oplus S^{-\star} \big) \boxdot Q \bigg) \\
&= \begin{bmatrix} S^H & \\ & (S^{-\star})^H \end{bmatrix} \begin{bmatrix} & R_\ell \\ R_\ell & \end{bmatrix} \begin{bmatrix} S & \\ & S^{- \star} \end{bmatrix} \boxdot Q^HR_k Q \\
&= \begin{bmatrix} & S^HR_\ell S^{-\star} \\ (S^{-\star})^HR_\ell S & \end{bmatrix} \boxdot R_k = R_{2\ell} \boxdot R_k = R_n
\end{align*}
where we used that $S^HR_\ell S^{-\star} = S^H R_\ell R_\ell S^{-H} R_\ell = S^H S^{-H} R_\ell = R_\ell$.

(b) If $S \in \Gl_{2\ell}(\CC)$ and $Q \in \Gl_k(\CC)$ are both perplectic, we obtain
\begin{align*} P^HR_nP &= \big(S^H \boxdot Q^H \big) \big( R_{2\ell} \boxdot R_k \big) \big( S \boxdot Q \big) = S^HR_{2\ell}S \boxdot Q^HR_kQ \\  &= R_{2\ell} \boxdot R_k = R_n. \end{align*}

(c) Whenever $P_1, P_2 \in \Gl_n(\CC)$ are both perplectic, then for $\tilde{P} := P_1P_2$ it holds that
$$\tilde{P}^HR_n \tilde{P} = (P_1P_2)^H R_n (P_1P_2) = P_2^H \big( P_1^HR_n P_1 \big) P_2 = P_2^H R_n P_2 = R_n,$$
so $\tilde{P}$ is perplectic. Analogously, the product of more than two perplectic matrices is perplectic.
\end{proof}

\section{A canonical form for $R_n$-normal matrices}
\label{s:perplectic}

In this section we develop a canonical form for nondefective (that is, diagonalizable) $R_n$-normal matrices under perplectic similarity. Our main result is Theorem \ref{thm:normal2X} prior to which we present several auxiliary results in order to keep the proof as compact as possible. We will use these results also in Section \ref{s:symplectic} where we consider the symplectic scalar product.

Our first observation is stated in Theorem \ref{thm:single-transform}. It shows that any diagonalizable per-Hermitian matrix $A \in \MM_n(\CC)$ is always perplectic similar to a perplectic sum of a diagonal matrix and a per-Hermitian matrix with only real eigenvalues.

\begin{theorem} \label{thm:single-transform}
Let $A \in \MM_n(\CC)$ be per-Hermitian and diagonalizable. Assume $\lambda_1, \ldots , \lambda_p \in \sigma(A)$ with multiplicities $\mathfrak{m}(A, \lambda_j) = m_j$, $j=1, \ldots , p$, are the distinct eigenvalues of $A$ with positive imaginary parts and set $m := m_1 + \cdots + m_p$.
Then there exists a perplectic matrix $P \in \Gl_n(\CC)$ such that
\begin{equation} P^{-1}AP = \begin{bmatrix} D & \\ & D^\star \end{bmatrix} \boxdot \hat{A} = \begin{bmatrix} D & & \\ & \hat{A} & \\ & & D^\star \end{bmatrix} \label{equ:transform1} \end{equation}
where $D = \lambda_1 I_{m_1} \oplus \cdots \oplus \lambda_p I_{m_p} \in \MM_m(\CC)$ and $\hat{A} \in \MM_{n-2m}(\CC)$ is per-Hermitian with only real eigenvalues\footnote{Note that for a diagonal matrix $D$ we obtain $D^\star$ simply through complex conjugation of the diagonal entries and reversing their order.}.
\end{theorem}

\begin{proof}
We prove the theorem by induction on the number $p$ of distinct eigenvalues with positive imaginary parts. To this end, recall \eqref{equ:spectrum} and notice that all eigenvalues of a per-Hermitian matrix beside the real axis appear in tuples $(\lambda, \overline{\lambda})$ with both having the same algebraic multiplicity.  If $A \in \MM_n(\CC)$ is per-Hermitian with only real eigenvalues (i.e. $p=0$) we are done (choose $P=I_n$).
So, assume the theorem holds for all per-Hermitian matrices with $p-1$ distinct eigenvalues with positive imaginary parts.

Let the distinct nonreal eigenvalues of $A \in \MM_n(\CC)$ with positive imaginary parts be given by $\lambda_1, \ldots , \lambda_p$. Moreover, let
$$U_0^{-1}AU_0 =   \big( \lambda_1 I_{m_1} \oplus \overline{\lambda}_1 I_{m_1} \big) \boxdot D_1 = \lambda_1 I_{m_1} \oplus D_1 \oplus \overline{\lambda}_1 I_{m_1} =: \tilde{D}$$ be a diagonalization of $A$ with $\mathfrak{m}(A,\lambda_1) = m_1$ being the multiplicity of $\lambda_1$ (implying $\mathfrak{m}(A, \overline{\lambda}_1) = m_1$) and some diagonal matrix $D_1 \in \MM_{n-2m_1}(\CC)$. According to Lemma \ref{lem:scalarprod1} (a) we know the form of the Hermitian matrix $U_0^HR_nU_0$ (because the columns of $U_0$ are eigenvectors of $A$). In fact, we have
$$ U_0^H R_n U_0 = \begin{bmatrix} 0 & S \\ S^H & 0 \end{bmatrix} \boxdot Q = \begin{bmatrix}  &  & S \\  & Q &  \\ S^H &  &  \end{bmatrix}$$
for some $S \in \MM_{m_1}(\CC)$ and some $Q = Q^H \in \MM_{n-2m_1}(\CC)$. Since $U_0$ and $R_n$ are nonsingular, so are $S$ and $Q$ (since $U_0^HR_nU_0$ is nonsingular, too).

Now we define the matrix  $U_{1} := R_{m_1} \oplus I_{n-2m_1} \oplus S^{-1} \in \Gl_n(\CC)$ and observe that
\begin{equation} U_1^H\left( U_0^H R_n U_0 \right) U_1 = \begin{bmatrix} & R_{m_1} \\ R_{m_1} & \end{bmatrix} \boxdot Q = R_{2m_1} \boxdot Q. \label{equ:transf1-proof1} \end{equation}
Due to the construction of $U_1$ we have $U_1^{-1} \tilde{D} U_1 = \tilde{D}$. As the matrix in \eqref{equ:transf1-proof1} is congruent to $R_n$, $Q$ must be congruent to $R_{n-2m_1}$. Therefore, there exists some $U_{2}' \in \Gl_{n-2m_1}(\CC)$ so that $(U_{2}')^HQU_{2}' = R_{n-2m_1}$. Defining $U_{2} := I_{m_1} \oplus U_{2}' \oplus I_{m_1}$ we obtain $U_{2}^H( R_{2m_1} \boxdot Q )U_2 = R_n$, so consequently the matrix $U_3 := U_0U_1U_2$ is perplectic. Now notice that

$$U_3^{-1}AU_3 = U_2^{-1} \tilde{D} U_2 = \begin{bmatrix} \lambda_1 I_{m_1} & \\ & \overline{\lambda}_1 I_{m_1} \end{bmatrix} \boxdot A'$$
where $A'  = (U_2')^{-1}D_1(U_2') \in \MM_{n-2m_1}(\CC)$ is now, in general, a full matrix. Since $U_3$ is perplectic, the matrix $U_3^{-1}AU_3$  is per-Hermitian according to Lemma \ref{lem:switch2}. Moreover, Remark \ref{rem:perp_properties} (c) reveals that $A'$ is a per-Hermitian matrix.
Now, the induction hypothesis applies to $A' \in \MM_{n-2m_1}(\CC)$ since $A'$ has only the $p-1$ nonreal eigenvalues $\lambda_2, \ldots , \lambda_p$. Thus, there exists some perplectic $P' \in \Gl_{n-2m_1}(\CC)$ such that $$\left( P' \right)^{-1}A'P' = \left( D_1' \oplus \left( D_1' \right)^\star \right) \boxdot \hat{A}$$
where $D_1' = \lambda_2 I_{m_2} \oplus \cdots \oplus \lambda_p I_{m_p}$, $(D_1')^\star = \overline{\lambda}_p I_{m_p} \oplus \cdots \oplus \overline{\lambda}_2 I_{m_2}$  and $\hat{A} \in \MM_{n-2m}(\CC)$ is per-Hermitian with only real eigenvalues. According to Lemma \ref{lem:perplectic} (a) the matrix $W := ( I_{m_1} \oplus I_{m_1}) \boxdot  P'  \in \Gl_n(\CC)$ is perplectic. From Lemma \ref{lem:perplectic} (c) we know that $P:= U_3W$ is perplectic and we obtain
$$P^{-1}AP = \begin{bmatrix} D & \\ & D^\star \end{bmatrix} \boxdot \hat{A} = \begin{bmatrix} D & & \\ & \hat{A} & \\ & & D^\star \end{bmatrix}$$
with $D := \lambda_1 I_{m_1} \oplus \cdots \oplus \lambda_p I_{m_p}$ and the matrix $\hat{A} \in \MM_{n-2m}(\CC)$ is per-Hermitian having only real eigenvalues.
\end{proof}

The next Theorem \ref{thm:sim-transf-1} states that the transformation carried out on $A$ in the proof of Theorem \ref{thm:single-transform} can - with some minor modifications - be extended to two commuting per-Hermitian matrices $A,B \in \MM_n(\CC)$. From this point of view, it is possible to simultaneously extract per-Hermitian matrices $\hat{A}, \hat{B}$ (of smaller size) with only real eigenvalues from commuting per-Hermitian $A,B \in \MM_n(\CC)$ via a perplectic similarity and the perplectic sum. We prove Theorem \ref{thm:sim-transf-1} as stated below although the result and its proof easily extend to any finite family of commuting per-Hermitian matrices.

\begin{theorem} \label{thm:sim-transf-1}
Let $A,B \in \MM_n(\CC)$ be per-Hermitian and diagonalizable and assume that $AB=BA$ holds. Then there exists some $s \in \mathbb{N}$, $0 \leq s \leq \lfloor n/2 \rfloor$ and some perplectic $P \in \Gl_n(\CC)$ such that
$$P^{-1}AP = \begin{bmatrix} D_A & \\ & D_A^\star \end{bmatrix} \boxdot \hat{A} \qquad P^{-1}BP = \begin{bmatrix} D_B & \\ & D_B^\star \end{bmatrix} \boxdot \hat{B}$$
where $D_A, D_B \in \MM_s(\CC)$ are diagonal and $\hat{A}, \hat{B} \in \MM_{n-2s}(\CC)$ are commuting per-Hermitian matrices with only real eigenvalues.
\end{theorem}

\begin{proof}
Assume that $\lambda_1, \ldots , \lambda_p \in \sigma(A)$ are the nonreal eigenvalues of $A$ with positive imaginary parts. Let $\mathfrak{m}(A, \lambda_j) = m_j \geq 1$ be the multiplicity of $\lambda_j$ and set $m:= m_1 + \cdots + m_p$. According to Theorem \ref{thm:single-transform}, there is some perplectic $P_1 \in \Gl_n(\CC)$ such that
$$P_1^{-1}AP_1 = \begin{bmatrix} D_{11} & \\ & D_{11}^\star \end{bmatrix} \boxdot A'$$
where $D_{11} = \lambda_1 I_{m_1} \oplus \cdots \oplus \lambda_p I_{m_p} \in \MM_m(\CC)$ and $A' \in \MM_{n-2m}(\CC)$ is per-Hermitian with only real eigenvalues. As $A$ and $B$ commute so do $P_1^{-1}AP_1$ and $P_1^{-1}BP_1$. Moreover, $P_1^{-1}AP$ and $P_1^{-1}BP_1$ are still per-Hermitian since $P_1$ is perplectic, cf. Lemma \ref{lem:switch2}. The commutativity implies that $P_1^{-1}BP_1$ has an analogous structure compared to $P_1^{-1}AP_1$, i.e. the form of $P_1^{-1}BP_1$ is determined as
\begin{equation} P_1^{-1}BP_1 = \begin{bmatrix} \tilde{B} & \\ & \tilde{B}^\star \end{bmatrix} \boxdot B' \label{equ:proof-transf-1} \end{equation}
where $\tilde{B} = \tilde{B}_1 \oplus \cdots \oplus \tilde{B}_p \in \MM_m(\CC)$ is block-diagonal with $\tilde{B}_j \in \MM_{m_j}(\CC)$ for $j=1, \ldots , p$, and $B' \in \MM_{n-2m}(\CC)$ is a per-Hermitian matrix commuting with $A'$. In view of Remark \ref{rem:perp_properties} (d) and Lemma \ref{lem:block-diag} we know that $B'$ and all matrices $\tilde{B}_j$, $j=1, \ldots , p$, are diagonalizable (since $B$ was diagonalizable and $P_1^{-1}BP_1$ is block-diagonal).

Now  let $\tilde{Q}_j^{-1} \tilde{B}_j \tilde{Q}_j = \tilde{D}_j$ be a diagonalization of $\tilde{B}_j$, $j=1, \ldots , p$. We define
$$\tilde{Q} := \tilde{Q}_1 \oplus \cdots \oplus \tilde{Q}_p \quad \textnormal{and} \quad Q := \left( \tilde{Q} \oplus \tilde{Q}^{-\star} \right) \boxdot I_{n-2m} \in \Gl_n(\CC).$$
 According to Lemma \ref{lem:perplectic} (a) and (c), $Q$ and $P_2 := P_1Q$ are perplectic. In addition, we obtain $P_2^{-1}AP_2 = P_1^{-1}AP_1$ (so $P_1^{-1}AP_1$ does not change under the similarity transformation with $Q$) and $P_2^{-1}BP_2 = ( D_{21} \oplus D_{21}^\star ) \boxdot B'$ where $D_{21} := \tilde{D}_1 \oplus \cdots \oplus \tilde{D}_p$ and $B'$ is per-Hermitian according to Remark \ref{rem:perp_properties} (c). Now we consider $A', B' \in \MM_{n-2m}(\CC)$ and set $m' := n-2m$. Recall that $A'$ has only real eigenvalues (by construction) which need not be the case for $B'$. Thus we may apply the same procedure to those matrices to get rid of the eigenvalues of $B'$ beside the real axis. So assume $\mu_1, \ldots , \mu_q$ are the distinct eigenvalues of $B'$ with positive imaginary parts having multiplicities $\mathfrak{m}(B', \mu_j) = r_j$ and set $r := r_1 + \cdots + r_q$ (note that $r \leq \lfloor m'/2 \rfloor$). According to Theorem \ref{thm:single-transform} there exists some perplectic matrix $P_1' \in \Gl_{m'}(\CC)$ such that
$$ \left( P_1' \right)^{-1}B'P_1' = \big( D_{22} \oplus D_{22}^\star \big) \boxdot \hat{B}$$
where $D_{22} = \mu_1 I_{r_1} \oplus \cdots \oplus \mu_q I_{r_q} \in \MM_r(\CC)$ and $\hat{B} \in \MM_{m'-2r}(\CC)$ is per-Hermitian with only real eigenvalues.
Since $A'$ and $B'$ commute and are both per-Hermitian, we conclude as in \eqref{equ:proof-transf-1} that
$$ \left( P_1' \right)^{-1}A'P_1' = \left( \breve{A} \oplus \breve{A}^\star \right) \boxdot \hat{A}$$ holds for some block-diagonal matrix $\breve{A} = \breve{A}_1 \oplus \cdots \oplus \breve{A}_q \in \MM_r(\CC)$ with $\breve{A}_j \in \MM_{r_j}(\CC)$ and some per-Hermitian $\hat{A} \in \MM_{m'-2r}(\CC)$. Using once more Lemma \ref{lem:block-diag} assume that $\breve{W}_j^{-1} \breve{A}_j \breve{W}_j = \breve{D}_j$ is a diagonalization of $\breve{A}_j$ for $j=1, \ldots , q$. We define
$$ \breve{W} := \breve{W}_1 \oplus \cdots \oplus \breve{W}_q \quad \textnormal{and} \quad W := \left( \breve{W} \oplus \breve{W}^{-\star} \right) \boxdot I_{m'-2r} \in \Gl_{m'}(\CC).$$
Then, according to Lemma \ref{lem:perplectic} (a) and (c), $W$ and $P_2' := P_1'W \in \Gl_{m'}(\CC)$ are perplectic. Moreover, we obtain $(P_2')^{-1}B'P_2' = (P_1')^{-1}B'P_1'$ (so the similarity transformation of $(P_1')^{-1}B'P_1'$ with $W$ does not change this matrix) and $(P_2')^{-1}A'P_2' = (D_{12} \oplus D_{12}^\star) \boxdot \hat{A}$ where we have set $D_{12} := \breve{D}_1 \oplus \cdots \oplus \breve{D}_q$. According to Lemma \ref{lem:perplectic} the matrices $P_3 :=  I_{2m}  \boxdot P_2' \in \Gl_n(\CC)$ and $P := P_1P_2P_3 \in \Gl_n(\CC)$ are both perplectic. So, finally we obtain
$$P^{-1}AP = \begin{bmatrix} D_A & \\ & D_A^\star \end{bmatrix} \boxdot \hat{A} \qquad P^{-1}BP = \begin{bmatrix} D_B & \\ & D_B^\star \end{bmatrix} \boxdot \hat{B}$$
with diagonal matrices $D_A = D_{11} \oplus D_{12} \in \MM_{s}(\CC)$ and $D_B = D_{21} \oplus D_{22} \in \MM_{s}(\CC)$ ($s=m+r$) and two commuting per-Hermitian matrices $A',B' \in \MM_{n-2(m+r)}(\CC)$ with only real eigenvalues.
\end{proof}

We refer to an even-sized matrix $A \in \MM_n(\CC)$ ($n=2m$) as being in $X$-form, if
\begin{equation} A= \begin{bmatrix} D_{11} & R_mD_{12} \\ R_mD_{21} & D_{22} \end{bmatrix} =  \left[ \tikz[baseline=+4mm]{ \draw[color=black] (0,1) -- (1,0); \draw[color=black] (0,0) -- (1,1); } \right] \label{equ:x-form} \end{equation}
where $D_{11}, D_{12}, D_{21}, D_{22} \in \MM_m(\CC)$ are diagonal. If $n$ is odd, we say that $A \in \MM_n(\CC)$ is in $X$-form if $A = A' \boxdot [ \alpha ]$ for some matrix $A \in \MM_{n-1}(\CC)$ in $X$-form as in \eqref{equ:x-form} and some scalar $\alpha \in \CC$.

We call a real matrix $A \in \MM_n(\RR)$ persymmetric, if $R_nA^TR_n = A$ holds. A persymmetric matrix $A$ with is additionally symmetric (i.e. $A=A^T$) is called bisymmetric.
Lemma \ref{lem:Z-transform} below shows how a real diagonal matrix can be unitarily transformed to a matrix in bisymmetric $X$-form. This result will be used in the proof of Theorem \ref{thm:simult-X-form}.

\begin{lemma} \label{lem:Z-transform}
Let $D \in \MM_{n}(\RR)$ be diagonal. If $n=2m$ is even we define the unitary (i.e. orthogonal) matrix
\begin{equation} Z := \frac{1}{\sqrt{2}} \begin{bmatrix} I_{m} & R_{m} \\ - R_{m} & I_{m} \end{bmatrix} \in \Gl_n(\RR). \label{equ:Z-matrix} \end{equation}
Then $Z^HDZ = Z^{-1}DZ$ is a real bisymmetric matrix in $X$-form. The same statement holds for the matrix $Z \boxdot [1] \in \Gl_n(\CC)$ if $n=2m+1$ is odd.
\end{lemma}

\begin{proof}
Assume $n=2m$ and write $D = D_1 \oplus D_2$ with two real, diagonal matrices $D_1,D_2 \in \MM_m(\CC)$. Then, a direct calculation shows that
\begin{equation} Z^HDZ = \frac{1}{2} \begin{bmatrix} D_1+R_mD_2R_m & D_1R_m-R_mD_2 \\ R_mD_1 - D_2R_m & R_mD_1R_m+D_2 \end{bmatrix} \in \MM_n(\RR). \label{equ:ZtoX} \end{equation}
Since $(D_1R_m-R_mD_2)^T = R_mD_1 - D_2R_m$, and both $D_1+R_mD_2R_m$ and $R_mD_1R_m+D_2$ are real and diagonal, it follows that $Z^HDZ$ is symmetric. Moreover, as $(D_1+R_mD_2R_m)^\star = R_mD_1R_m+D_2$ is diagonal and both $D_1R_m-R_mD_2$, $ R_mD_1 - D_2R_m$ are real and do only have nonzero entries along their anti-diagonals, $Z^HDZ$ is persymmetric. Therefore, $Z^HDZ$ is real and bisymmetric. The statement follows analogously for $Z \boxdot [1]$ in case $n=2m+1$ is odd.
\end{proof}

For $n=2m$ it follows directly from Lemma \ref{lem:Z-transform} and \eqref{equ:ZtoX} (choosing $D_1 = +I_m$ and $D_2 = -I_m$) that
\begin{equation} Z^H \begin{bmatrix} +I_m & \\ & -I_m \end{bmatrix} Z = \begin{bmatrix} & R_m \\ R_m & \end{bmatrix} = R_{n} \label{equ:ZtoR} \end{equation}
for the matrix $Z \in \Gl_{n}(\RR)$ defined in \eqref{equ:Z-matrix}. Considering $Z \boxdot [1]$ in the case where $n$ is odd yields the similar result. With Lemma \ref{lem:Z-transform} we may now prove the simultaneous transformation of two per-Hermitian matrices with only real eigenvalues to $X$-form as stated in Theorem \ref{thm:simult-X-form}. Our main theorem on the perplectic transformation of $R_n$-normal matrices to $X$-form  (Theorem \ref{thm:normal2X}) will then easily follow from Theorem \ref{thm:sim-transf-1} and Theorem \ref{thm:simult-X-form}.

\begin{theorem} \label{thm:simult-X-form}
Let $A,B \in \MM_n(\CC)$ be per-Hermitian and diagonalizable and assume that $AB=BA$ holds. Moreover, suppose that $A,B$ have exclusively real eigenvalues. Then, there is a perplectic matrix $P \in \Gl_n(\CC)$ such that
$$P^{-1}AP =: X_A =  \left[ \tikz[baseline=+4mm]{ \draw[color=black] (0,1) -- (1,0); \draw[color=black] (0,0) -- (1,1); } \right], \quad P^{-1}BP =: X_B =  \left[ \tikz[baseline=+4mm]{ \draw[color=black] (0,1) -- (1,0); \draw[color=black] (0,0) -- (1,1); } \right]$$
where $X_A,X_B \in \MM_n(\RR)$ are real bisymmetric matrices in $X$-form.
\end{theorem}

\begin{proof}
Assume $\lambda_1, \ldots , \lambda_p \in \RR$ are the distinct eigenvalues of $A$ with multiplicities $\mathfrak{m}(A, \lambda_j)=m_j \geq 1$ for $j=1, \ldots ,p$, and $\mu_1, \ldots , \mu_q \in \RR$ are the distinct eigenvalues of $B$. Assume further that $T \in \Gl_n(\CC)$ simultaneously diagonalizes $A$ and $B$, i.e. $A_1 := T^{-1}AT$ and $B_1 := T^{-1}BT$ are both diagonal (cf. Lemma \ref{lem:simult-diag}). In particular, suppose  $A_1 = \lambda_1 I_{m_1} \oplus \cdots \oplus \lambda_p I_{m_p}$ and $B_1 = D_1 \oplus \cdots \oplus D_p$ where the matrices  $D_j \in \MM_{m_j}(\RR)$, $j=1, \ldots , p$, are diagonal. Notice that each eigenvalue $\mu_j$ of $B$ appears $\mathfrak{m}(B, \mu_j)$ times on the diagonal of $B_1$ (in some of the matrices $D_1, \ldots , D_p$). Without loss of generalization we can assume that eigenvalues of $D_j$ with higher multiplicities are grouped together on the diagonal of $D_j$.
So assume that $r_j \geq 1$ distinct eigenvalues of $B$ appear in $D_j$ for each $j=1, \ldots , p$. We denote these by $\mu_{j,1}, \ldots , \mu_{j,{r_j}}$. Do not overlook that each $\mu_{j,k}$ is equal to some eigenvalue among $\mu_1, \ldots , \mu_q$. Let their multiplicities in $D_j$ be $\mathfrak{m}(D_j, \mu_{j,k}) = m_{j,k}$ for $k=1, \ldots , r_j$. Thus, w.\,l.\,o.\,g. we assume that $D_j$ can be expressed as\footnote{Otherwise, we may reorder the columns of $T$ to produce this form.}
$$D_j = \mu_{j,1} I_{m_{j,1}} \oplus \cdots \oplus \mu_{j,r_j} I_{m_{j,r_j}}, \qquad j=1, \ldots , p,$$
with $m_{j,1} + \cdots + m_{j,r_j} = m_j$.

Recall that the columns of $T$ are eigenvectors of both $A$ and $B$. Thus, we may determine the form of the Hermitian matrix $T^HR_nT$ according to Lemma \ref{lem:scalarprod1} (a).
In fact, we have $T^HR_nT = S_1 \oplus \cdots \oplus S_p$ for Hermitian matrices $S_j \in \MM_{m_j}(\CC)$ following from the structure of $A_1$. As $T^HR_nT$ is nonsingular, so is each $S_j$, $j=1, \ldots, p$. In addition, the structure of $B_1$ (in particular, the form of $D_1, \ldots, D_p$) implies that each $S_j$ has to be block-diagonal, i.e.
$$ S_j = \begin{bmatrix} S_{j,1} & & \\ & \ddots & \\ & & S_{j,r_j} \end{bmatrix}, \qquad j=1, \ldots , p,$$
where $S_{j,k} \in \MM_{m_{j,k}}(\CC)$, $j=1, \ldots , p$ and $k=1, \ldots , r_j$, are nonsingular and Hermitian.  It follows from Theorem \ref{thm:sylvester-inertia} that there exist matrices $Q_{j,k} \in \Gl_{m_{j,k}}(\CC)$ so that $Q_{j,k}^H S_{j,k} Q_{j,k} = \textnormal{diag}(+1, \ldots , -1, \ldots)$ is the inertia matrix corresponding to $S_{j,k}$. Thus we define $Q := Q_1 \oplus \cdots \oplus Q_p \in \Gl_n(\CC)$ with $Q_j := Q_{j,1} \oplus \cdots \oplus Q_{j,r_j} \in \Gl_{m_j}(\CC)$.
Now we consider the transformations $Q^{-1}A_1Q$, $Q^{-1}B_1Q$ and $Q^H(T^HR_nT)Q$ and make the following important observations:

\begin{enumerate}[(a)]
\item It follows from the form of $A_1$ and $B_1$ and the form of $Q$ that the transformations  $Q^{-1}A_1Q$ and $Q^{-1}B_1Q$ do not change $A_1$ and $B_1$, respectively. Therefore we have $Q^{-1}A_1Q = A_1$ and $Q^{-1}B_1Q = B_1$.
\item The matrix $Q^H(T^HR_nT)Q$ is diagonal with only $+1$ and $-1$ entries along its diagonal. As it is congruent to $R_n$, there are exactly $\lceil n/2 \rceil$ entries equal to $+1$ and $\lfloor n/2 \rfloor$ entries equal to $-1$ appearing.
\end{enumerate}
Now we consider the cases of even and odd $n$ separately. First assume that $n$ is even.
Then there is a permutation matrix $W \in \Gl_n(\RR)$ such that
\begin{equation} W^H \big( Q^H T^HR_n TQ \big) W = \begin{bmatrix} +I_{n/2} & \\ & -I_{n/2} \end{bmatrix}. \label{equ:plus-minus-matrix} \end{equation}
Note that $W^H = W^T = W^{-1}$. As $W$ is a permutation, $A_2 := W^{T}A_1W$ and $B_2 := W^{T}B_1W$ remain to be real and diagonal matrices. Now Lemma \ref{lem:Z-transform} applies to $A_2,B_2$.
Recall that, for $Z \in \Gl_n(\RR)$ as in \eqref{equ:Z-matrix}, we have shown in \eqref{equ:ZtoR} that $Z^H (+I_{n/2} \oplus -I_{n/2})Z = R_n$ holds. Therefore, according to \eqref{equ:plus-minus-matrix}, $P:=TQWZ$ is perplectic. Moreover, by Lemma \ref{lem:Z-transform}, $Z^{-1}A_2Z =:X_A$ and $Z^{-1}B_2Z =:X_B$ are bisymmetric and in $X$-form since $A_2,B_2 \in \MM_n(\RR)$ are real and diagonal. In consequence, $P^{-1}AP = X_A$ and $P^{-1}BP =X_B$ are perplectic transformations of $A$ and $B$ to bisymmetric $X$-form and the statement is proven for even $n$.
In case $n$ is odd, the permutation $W$ can be chosen such that the matrix in \eqref{equ:plus-minus-matrix} has the form $(+I_{\lfloor n/2 \rfloor} \oplus -I_{ \lfloor n/2 \rfloor}) \boxdot [1]$. The matrix $Z$ from \eqref{equ:Z-matrix} can be replaced by $Z \boxdot [1]$ (see the discussion subsequent to \eqref{equ:ZtoR}) and the statement follows analogously.
\end{proof}

We are now ready to prove the main theorem of this section. To this end, we introduce for any $A \in \MM_n(\CC)$ the matrices
\begin{equation}
S_A := \frac{1}{2} \left( A + A^\star \right) \quad \textnormal{and} \quad K_A := \frac{i}{2} \left( A - A^\star \right).
\label{equ:splitting}
\end{equation}
As $(A^\star)^\star = A$ and $(A+A')^\star = A^\star + (A')^\star$ always holds for any $A \in \MM_n(\CC)$, it is easily seen that $S_A$ and $K_A$ are both per-Hermitian and that $A = S_A - iK_A$ holds.
The crucial observation used for the proof of Theorem \ref{thm:normal2X} is that $S_A$ and $K_A$ commute whenever $A$ is $R_n$-normal (i.e. $AA^\star = A^\star A$).

\begin{theorem} \label{thm:normal2X}
Let $A \in \MM_n(\CC)$ be diagonalizable and $R_n$-normal. Then there exists a perplectic matrix $P \in \Gl_n(\CC)$ such that
\begin{equation} P^\star A P = X_A =  \left[ \tikz[baseline=+4mm]{ \draw[color=black] (0,1) -- (1,0); \draw[color=black] (0,0) -- (1,1); } \right] \label{equ:XformThm} \end{equation}
is a matrix in $X$-form.
\end{theorem}

\begin{proof}
Let $A \in \MM_n(\CC)$ be $R_n$-normal and diagonalizable. We express $A$ as $S - iK$ for $S := S_A,K := K_A \in \MM_n(\CC)$ as defined in \eqref{equ:splitting}. Notice that, since $A$ is diagonalizable, so is $A^\star$. Moreover, as $A$ and $A^\star$ commute, they are simultaneously diagonalizable, cf. Lemma \ref{lem:simult-diag}. Consequently, both $S$ and $K$ are diagonalizable. Now we apply Theorem \ref{thm:sim-transf-1} to $S$ and $K$. Thus there exists some $s \in \mathbb{N}$, $0 \leq s \leq \lfloor n/2 \rfloor$ and a perplectic matrix $P_1 \in \Gl_n(\CC)$ such that $S_1 := P_1^{-1}SP_1$ and $K_1 := P_1^{-1}KP_1$ have the form
$$S_1 = \begin{bmatrix} D_S & \\ & D_S^\star \end{bmatrix} \boxdot S_1' \quad \textnormal{and} \quad K_1 = \begin{bmatrix} D_K & \\ & D_K^\star \end{bmatrix} \boxplus K_1'$$
where $D_S,D_K, D_S^\star, D_K^\star \in \MM_{s}(\CC)$ are diagonal and $S_1', K_1' \in \MM_{n-2s}(\CC)$ are both per-Hermitian with only real eigenvalues. As $S$ and $K$ commute, we have $S_1K_1 = K_1S_1$ and therefore $S_1' K_1' = K_1'S_1'$.

According to Theorem \ref{thm:simult-X-form} there is some perplectic $P_2' \in \Gl_{n-2s}(\CC)$ such that $(P_2')^{-1}S_1'P_2' =: X_S$ and $(P_2')^{-1}K_1'P_2' =: X_K$ are matrices in $X$-form. Since the matrix $P_2 := I_{2s} \boxdot P_2' \in \Gl_n(\CC)$  is perplectic by Lemma \ref{lem:perplectic} (a), $P:= P_1P_2$ is perplectic according to Lemma \ref{lem:perplectic} (c). Now the matrices
\begin{equation} P^{-1}SP = \begin{bmatrix} D_S & & \\ & X_S & \\ & & D_S^\star \end{bmatrix} \quad \textnormal{and} \quad  P^{-1}KP = \begin{bmatrix} D_K & & \\ & X_K & \\ & & D_K^\star \end{bmatrix} \label{equ:normal2X} \end{equation}
are both in $X$-form. Moreover, $P^{-1}AP = P^{-1}SP - i \big( P^{-1}KP \big) =: X_A$ is a matrix in $X$-form and the statement is proven.
\end{proof}

\begin{remark}
Notice that, in the proof of Theorem \ref{thm:normal2X}, the matrix $P^{-1}AP$ has an $X$-form where the anti-diagonal is not completely equipped with nonzero entries.
\end{remark}

\subsection{Genericity of the X-form for $R_n$-normal matrices}
\label{ss:genericity}

From now on, we denote the set of all $R_n$-normal matrices by $\mathcal{N}(R_n)$.
According to \cite[Thm.\,8]{CruzSalten} the set of diagonalizable $R_n$-normal matrices is dense in $\mathcal{N}(R_n)$. This means that for any $A \in \mathcal{N}(R_n)$ and any given $\varepsilon > 0$ there is some diagonalizable $A' \in \mathcal{N}(R_n)$ with $\Vert A - A' \Vert_2 < \varepsilon$. In this section we will prove that the set of $R_n$-normal matrices with $n$ distinct eigenvalues is dense in $\mathcal{N}(R_n)$. For this purpose we use the perplectic transformation to $X$-form established in Theorem \ref{thm:normal2X}. As a consequence, we will show in Theorem \ref{thm:genericity} that this fact justifies us to refer to the $X$-form as being generic for $R_n$-normal matrices. To prove these results, the following Proposition \ref{prop:commute} will be helpful.

\begin{proposition} \label{prop:commute}
Let $A \in \MM_{n}(\CC)$ be $R_{n}$-normal and diagonalizable. Then there exists some $R_{n}$-normal matrix $M \in \MM_{n}(\CC)$ with $n$ distinct eigenvalues such that $M$ commutes with $A$ and $A^\star$.
\end{proposition}

\begin{proof}
Let $A \in \MM_{n}(\CC)$ be $R_{n}$-normal and diagonalizable. According to Theorem \ref{thm:normal2X} there is some perplectic matrix $P \in \Gl_{n}(\CC)$ such that $P^{-1}AP = X_A$ is in $X$-form, see \eqref{equ:XformThm}. If $n=2 \ell$ is even, we may express $X_A$ as
$$X_A = N_1 \boxdot \cdots \boxdot N_\ell = \big( N_1 \boxdot N_2 \big) \boxdot N_3) \boxdot \cdots \big) \boxdot N_\ell$$
with $2 \times 2$ matrices $N_i$, $i=1, \ldots , \ell$ (recall also \eqref{equ:perpsum-notation}). That is
$$ X_A = \begin{bmatrix} n_{11}^1 & & & & & & & n_{12}^1 \\ & n_{11}^2 & & & & & n_{12}^2 & \\ & & \ddots & & & \iddots & & \\ & & & n_{11}^\ell & n_{12}^\ell & & & \\ & & & n_{21}^\ell & n_{22}^\ell & & & \\ & & \iddots & & & \ddots & & \\ & n_{21}^2 & & & & & n_{22}^2 & \\ n_{21}^1 & & & & & & & n_{22}^1 \end{bmatrix}$$
where
$$N_1 = \begin{bmatrix} n_{11}^1 & n_{12}^1 \\ n_{21}^1 & n_{22}^2 \end{bmatrix}, \; N_2 = \begin{bmatrix} n_{11}^2 & n_{12}^2 \\ n_{21}^2 & n_{22}^2 \end{bmatrix}, \ldots , N_\ell = \begin{bmatrix} n_{11}^\ell & n_{12}^\ell \\ n_{21}^\ell & n_{22}^\ell \end{bmatrix}.$$
In case $n= 2\ell +1$ is odd we may express $X_A$ in the same form where now $\ell = \lceil n/2 \rceil$ and $N_\ell \in \CC$ is just a scalar. As $P$ is perplectic, $X_A$ is again $R_n$-normal according to Lemma \ref{lem:switch2}. For the rest of the proof we confine ourselves to the case $n = 2\ell$. When $n$ is odd, an analogous reasoning gives the same results.

At first, we consider the explicit form of $X_A^\star X_A$ and $X_AX_A^\star$. With $R_n = R_2 \boxdot \cdots \boxdot R_2$ ($\ell$ factors) we find $X_A^\star = \boxdot_{i=1}^\ell N_i^\star$ and therefore
\begin{equation} X_A^\star X_A = \boxdot_{i=1}^\ell R_2N_i^HR_2N_i \quad \textnormal{and} \quad X_A X_A^\star = \boxdot_{i=1}^\ell N_iR_2N_i^HR_2 \label{equ:R2normal} \end{equation}
using \eqref{equ:perpsum_1} and noting that $R_2^{-1}=R_2$.
As both matrices are equal, we may compare the terms in \eqref{equ:R2normal} and see that each $N_j$ is $R_2$-normal for all $j=1, \ldots , \ell$.
Now recall additionally that $X_A$ is similar to $N_1 \oplus \cdots \oplus N_\ell$ according to Remark \ref{rem:perp_properties} (d), so each $N_j$ is diagonalizable (since $A$ and consequently $X_A$ are diagonalizable).
For scalars $\alpha_k \in \CC$, $k=1, \ldots , \ell$, whose desired property will be clear in a moment, we define matrices $M_k' \in \MM_{2}(\CC)$ according to the following rules:
\begin{enumerate}[(a)]
\item If $N_k$ has two identical eigenvalues we define $M_k' := \textnormal{diag}( \alpha_k, 1 + \alpha_k)$ which has two distinct eigenvalues $\alpha_k$ and $1+ \alpha_k$. (Note that, in this case, $N_k$ is just a multiple of the identity $I_2$.)
\item If $N_k$ has two distinct eigenvalues, we define $M_k' := \alpha_k N_k$ for $\alpha_k \neq 0$ which also has two distinct eigenvalues.
\end{enumerate}
Now we define $$M' := (M_1' \boxdot M_2') \boxdot M_3') \boxdot \cdots ) \boxdot M_\ell' \in \MM_{n}(\CC).$$
 As $M'$ is similar to $M_1' \oplus \cdots \oplus M_\ell'$ according to Remark \ref{rem:perp_properties} (d), the eigenvalues of $M'$ are those of $M_1', M_2', \ldots , M_\ell'$. The crucial observation is that
 we are always able to choose $\alpha_1, \ldots , \alpha_n$ in such a way that $M'$ has $n$ distinct eigenvalues.

Now each $M_k' \in \MM_2(\CC)$ is $R_2$-normal. To accept this, note that any diagonal matrix $M_k'$ (as in (a) above) is always $R_2$-normal and any scalar multiple $M_k' := \alpha_k N_k$ of an $R_2$-normal matrix (as in (b) above) remains to be $R_2$-normal. With the same calculations as in \eqref{equ:R2normal} it is easy to see that $M'$ is $R_{n}$-normal. Moreover, by construction, each $M_k'$ commutes with each $N_k$ and with $N_k^\star = R_2 N_k^H R_2$. Therefore, $M'$ commutes with $X_A$ and with $X_A^\star = N_1^\star \boxdot \cdots \boxdot N_\ell^\star$. This implies that $M := PM'P^{-1}$, which is $R_n$-normal according to Lemma \ref{lem:switch2} and whose $n$ eigenvalues are distinct, commutes with $A = PX_AP^{-1}$ and $A^\star = PX_A^\star P^{-1}$. This proves the statement.
\end{proof}

The following Lemma \ref{lem:polynomial} is important for the proof of Theorem \ref{thm:approx-distinct}. It can be found in \cite[Ex.\,2\,(ii)]{CruzSalten} or \cite[Sec.\,7]{WeyrForm} and is stated here without proof.

\begin{lemma} \label{lem:polynomial}
There exists a polynomial
\begin{equation} p(x_{11},x_{12}, \ldots , x_{nn}) \in \CC[x_{11}, \ldots , x_{nn}] \label{equ:polynomial} \end{equation}
 in $n^2$ variables $x_{jk}$ such that, for any $A = [a_{ij}]_{i,j} \in \MM_{n}(\CC)$,
\begin{equation} p(A) := p(a_{11}, a_{12}, \ldots , a_{nn}) = 0 \label{equ:polynomial1} \end{equation}
if and only if $A$ has at least one multiple eigenvalue\footnote{In other words, the set of all matrices with at least one multiple eigenvalues is an algebraic variety.}.
\end{lemma}

In other words, Lemma \ref{lem:polynomial} states that one can tell if a given matrix $A \in \MM_{n}(\CC)$ has a multiple eigenvalue by evaluating $p$ in \eqref{equ:polynomial} at $A$, that is, at $(a_{11}, a_{12}, \ldots , a_{nn})$. For the proof of the following Theorem \ref{thm:approx-distinct} notice that, if $N,M \in \MM_{n}(\CC)$ are two $R_n$-normal matrices and $M$ commutes with $N$ and $N^\star$, then $zN + M$ is also $R_{n}$-normal for any $z \in \CC$. This is easily seen since
\begin{equation} \begin{aligned}
(zN+M)^\star(zN+M) &= \overline{z}z N^\star N + \overline{z} N^\star M + zM^\star N + M^\star M \\
(zN+M)(zN+M)^\star &= z \overline{z} NN^\star + z N M^\star + \overline{z} M N^\star + MM^\star
\end{aligned} \label{equ:normalMN} \end{equation}
using the $R_{n}$-normality of $M$ and $N$ and the commutativity of $M$ and $N^\star$ (moreover, note that $M^\star N = (N^\star M)^\star$ and that $NM^\star = (M N^\star)^\star$).

We are now able to prove our first main result of this section.
Theorem \ref{thm:approx-distinct} states that the set of $R_n$-normal matrices with $n$ distinct eigenvalues is dense in $\mathcal{N}(R_n)$.

\begin{theorem} \label{thm:approx-distinct}
Let $A \in \MM_{n}(\CC)$ be $R_n$-normal and let $\varepsilon > 0$ be given. Then there exists some $R_n$-normal matrix $\hat{A} \in \MM_{n}(\CC)$ with $n$ distinct eigenvalues such that $\Vert A - \hat{A} \Vert_2 < \varepsilon$.
\end{theorem}

\begin{proof}
Let $A \in \MM_{n}(\CC)$ be $R_n$-normal and let $\varepsilon > 0$ be given. According to \cite[Thm.\,7]{CruzSalten}, there exists some diagonalizable, $R_n$-normal $A' \in \MM_{n}(\CC)$ such that $\Vert A - A' \Vert_2 < \varepsilon/2$. Moreover, according to Proposition \ref{prop:commute}, there exists some $R_n$-normal $N \in \MM_{n}(\CC)$ with $n$ distinct eigenvalues that commutes with $A'$ and $(A')^\star$. Then, for any $z \in \CC$, the matrix $M(z) := zA' + N$ is $R_n$-normal (see \eqref{equ:normalMN} and the discussion above). Let $p(x_{11}, \ldots , x_{nn})$ be the polynomial from Lemma \ref{lem:polynomial}. Using the notation from \eqref{equ:polynomial1}, we now consider $\tilde{p}(z) := p(M(z))$ as a polynomial in the single variable $z$. Observe that $\tilde{p}(0) = p(M(0)) = p(N) \neq 0$ since $N$ has $n$ distinct eigenvalues. Therefore,  $\tilde{p} \neq 0$ is not the zero-polynomial.

Recall from Lemma \ref{lem:polynomial} that $\tilde{p}(z_0) \neq 0$ is a sufficient and necessary condition for $M(z_0)$ to have $n$ distinct eigenvalues. Now, since $\tilde{p}(z)$ has only a finite number of roots, we conclude that almost every matrix $M(z) = zA'+N$ has $n$ distinct eigenvalues. Therefore, $A' + cN = c(c^{-1}A'+N)$ also has $n$ distinct eigenvalues for almost every $c \in \CC$, $c \neq 0$. As $A' + cN = c M(c^{-1})$ is a scalar multiple of a $R_n$-normal matrix, it is $R_n$-normal, too.

Now choose some $c_0 \in \CC$ with $| c_0 | < \varepsilon /(2 \Vert N \Vert_2)$ such that $\tilde{p}(c_0^{-1}) \neq 0$ and define $\hat{A} := A' + c_0N$. Then $\hat{A}$ is $R_n$-normal and has $n$ distinct eigenvalues. Moreover,
$$ \Vert A' - \hat{A} \Vert_2 = \Vert A' - \big( A' + c_0 N \big) \Vert_2 = | c_0 | \cdot \Vert N \Vert_2 < \frac{\varepsilon}{2}$$
and it follows that
$$ \Vert A - \hat{A} \Vert_2 \leq \Vert A - A' \Vert_2 + \Vert A' - \hat{A} \Vert_2 < \frac{\varepsilon}{2} + \frac{\varepsilon}{2} = \varepsilon.$$
Thus we have shown that there exists some $R_n$-normal matrix with $n$ distinct eigenvalues with distance less than $\varepsilon$ from $A$ and the proof is complete.
\end{proof}

Recall that $\MM_{n}(\CC)$ can be considered as a topological space with basis $B_R(A) = \lbrace A' \in \MM_{n}(\CC) \, : \, \Vert A-A' \Vert_2 < R \rbrace$ for $A \in \MM_{n}(\CC)$ and $R \in \RR$, $R > 0$ (see, e.g. \cite[Sec.\,11.2]{Hartfiel}). The set $\mathcal{N}(R_n)$ of $R_n$-normal matrices can thus be interpreted as a topological space on its own equipped with its subspace topology \cite[Sec.\,1.5]{top}. Per definition, a subset $\mathcal{S}$ of $\mathcal{N}(R_n)$ is open if $\mathcal{S}$ is the intersection of $\mathcal{N}(R_n)$ with some open subset of $\MM_{n}(\CC)$. It is well-known that the set $\mathcal{E}' \subset \MM_{n}(\CC)$ of matrices with $n$ distinct eigenvalues is open in the topological space $\MM_{n}(\CC)$, cf. \cite[Thm.\,11.5]{Hartfiel}. Thus, $\mathcal{E} := \mathcal{E}' \cap \mathcal{N}(R_n)$ is an open subset of $\mathcal{N}(R_n)$ (with respect to its subspace topology) whose closure is all of $\mathcal{N}(R_n)$ according to the density established in Theorem \ref{thm:approx-distinct}.  Consequently, the complement of $\mathcal{E}$ in $\mathcal{N}(R_n)$ is a nowhere dense subset in $\mathcal{N}(R_n)$. Since all matrices in $\mathcal{E}$ have pairwise distinct eigenvalues, they are diagonalizable and thus Theorem \ref{thm:normal2X} applies. It is shown in \cite[Thms.\,5,6]{CruzSalten} that the set of matrices with $n$ distinct eigenvalues is dense in the sets of all per-Hermitian, perskew-Hermitian and perplectic matrices, too. Thus, with exactly the same reasoning as above the topological analysis carried out for $\mathcal{N}(R_n)$ applies to these sets of matrices. This leads us to the second main result of this section stated in Theorem \ref{thm:genericity}.

\begin{theorem} \label{thm:genericity}
With $\mathcal{S} \subseteq \MM_n(\CC)$ being one of the sets of all per-Hermitian, perskew-Hermitian, perplectic or $R_n$-normal matrices, the following is true:
the set of all matrices $A \in \mathcal{S}$ for which a perplectic matrix $P \in \Gl_n(\CC)$ exists such that $P^{-1}AP = X_A$ is in $X$-form is open and dense. Consequently, the set of all matrices $A \in \mathcal{S}$ for which such a perplectic similarity transformation does not exist is nowhere dense in $\mathcal{S}$.
\end{theorem}

According to \cite{top} an open set of a topological space whose interior is dense can reasonably considered to be large. In turn, its complement is small from a topological point of view. This justifies calling the $X$-form established in Theorem \ref{thm:normal2X} under perplectic similarity \emph{generic} for the sets of per-Hermitian, perskew-Hermitian, perplectic and $R_n$-normal matrices.

\section{A generic canonical form for $J_{2n}$-normal matrices}
\label{s:symplectic}

From here on we consider the indefinite scalar product $(x,y) \mapsto [x,y]_{J_{2m}} := x^HJ_{2m}y$ defined on $\CC^{2m} \times \CC^{2m}$ for the skew-Hermitian matrix
$$J_{2m} = \begin{bmatrix} & I_m \\ -I_m & \end{bmatrix} \in \Gl_{2m}(\RR).$$
The eigenvalues of $J_{2m}$ are $+i$ and $-i$ with multiplicities $\mathfrak{m}(J_{2n},+i) = $ and $\mathfrak{m}(J_{2n},-i)=m$. It is common to refer to $J_{2m}$-self/skewadjoint matrices as skew-Hamiltonian and Hamiltonian, respectively\footnote{Sometimes these matrices are called $J_{2m}$-Hermitian and $J_{2m}$-skew-Hermitian, see \cite{Mack07}.}. Matrices that are $J_{2m}$-unitary are in general called symplectic. For the set of all $J_{2m}$-normal matrices we will use the notation $\mathcal{N}(J_{2m})$.
In this section we use the results obtained in Section \ref{s:perplectic} to derive a generic canonical form for $J_{2m}$-normal matrices. To this end, we set $n := 2m$ for the whole section whenever $n$ is not further specified.

First consider the $2m \times 2m$ unitary matrix
\begin{equation} U := \begin{bmatrix} I_m & \\ & -iR_m \end{bmatrix} \in \Gl_{2m}(\CC). \label{equ:U-transform} \end{equation}
A straight forward calculation shows that $U^H(iR_n)U = J_{2m}$. Over $\CC^{n} \times \CC^{n}$, the classes of per-Hermitian, perskew-Hermitian, perplectic and $R_n$-normal matrices
can now be related in a natural way to skew-Hamiltonian, Hamiltonian, symplectic and $J_{2m}$-normal matrices. These relations are summarized in Lemma \ref{lem:perp2symp}.

\begin{lemma} \label{lem:perp2symp}
Let $A \in \MM_{2m}(\CC)$ and $n=2m$. The following relations hold for the classes of matrices introduced in Section \ref{s:intro} with respect to the indefinite scalar products $[ \cdot, \cdot]_{J_{2m}}$ and $[ \cdot, \cdot]_{R_n}$.
\begin{enumerate}[(a)]
\item If $A$ is skew-Hamiltonian, then $A' := UAU^H$ is per-Hermitian. In turn, $U^HAU$ is skew-Hamiltonian whenever $A$ is per-Hermitian.
\item If $A$ is Hamiltonian, then $A' := UAU^H$ is perskew-Hermitian. On the other hand, $U^HAU$ is Hamiltonian whenever $A$ is perskew-Hermitian.
\item If $A$ is symplectic, then $A' := UAU^H$ is perplectic. In turn, $U^HAU$ is symplectic whenever $A$ is perplectic.
\item If $A$ is $J_{2m}$-normal, then $A' := UAU^H$ is $R_n$-normal. On the other hand $U^HAU$ is $J_{2m}$-normal whenever $A$ is $R_n$-normal.
\end{enumerate}
\end{lemma}

\begin{proof}
All four cases can be proved by direct calculations using $UJ_{2m}U^H = iR_n$, i.e. $UJ_{2m} = (iR_n)U$ and the equations specifying the structures in (a), (b), (c) and (d).

(a) Assume $A \in \MM_n(\CC)$ is per-Hermitian, that is $R_nA^HR_n = A$ holds. Then $U^HAU$ is skew-Hamiltonian since
\begin{align*}
J_{2m}^{-1} \big( U^HAU \big)^HJ_{2m} &= J_{2m}^{-1} \big( U^HA^HU \big) J_{2m} = J_{2m}^{-1} \big( U^H \big( R_n A R_n \big) U \big) J_{2m} \\
&= \big( UJ_{2m} \big)^H R_n A R_n \big(XJ_{2m} \big) \\
&= -i^2 U^H \big( R_n R_n \big) A \big( R_n R_n \big) U = U^HAU
\end{align*}
where we used that $J_{2m}^{-1} = J_{2m}^H = J_{2m}^T$ and $R_nR_n = I_n$. On the other hand, if $A \in \MM_{2m}(\CC)$ is skew-Hamiltonian, i.e. $J_{2m}^HA^HJ_{2m} = A$, then $UAU^H$ is per-Hermitian because
\begin{align*}
R_n \big( UAU^H \big)^H R_n &= R_n UA^HU^H R_n = R_n U J_{2m} A J_{2m}^H U^H R_n \\
&= R_n \big( iR_nU \big) A \big(iR_n U \big)^H R_n \\
&= -i^2 \big( R_n R_n \big) UAU^H \big(R_n R_n \big) = UAU^H.
\end{align*}
The proofs for (b), (c) and (d) proceed in a similar manner.
\end{proof}

In other words, Lemma \ref{lem:perp2symp} states that there exists a one-to-one correspondence (i.e. a bijection) between the sets of $B$-selfadjoint, $B$-skewadjoint, $B$-unitary and $B$-normal matrices for $B=J_{2m}$ and $B=R_n$. For the proofs of this section we will be switching to and fro between these sets via the similarity with $U$ in \eqref{equ:U-transform}.

We say a matrix $A \in \MM_{2m}(\CC)$ is in four-diagonal-form, if
\begin{equation} A = \begin{bmatrix} D_{11} & D_{12} \\ D_{21} & D_{22} \end{bmatrix} = \left[ \tikz[baseline=+4mm]{ \draw[color=black] (0,1) -- (1,0); \draw[color=black] (0,0.5) -- (0.5,0); \draw[color=black] (0.5,1) -- (1,0.5); } \right] \label{equ:4diagonals} \end{equation}
where $D_{11}, D_{21}, D_{12}, D_{22} \in \MM_m(\CC)$ are diagonal matrices.

Based on the results from Section \ref{s:perplectic} we may now establish a canonical form of $J_{2m}$-normal diagonalizable matrices under symplectic similarity transformations.

\begin{theorem} \label{thm:normal2-4diags}
Let $A \in \MM_{2m}(\CC)$ be $J_{2m}$-normal and diagonalizable. Then there exists some symplectic matrix $S \in \Gl_{2m}(\mathbb{C})$ such that
\begin{equation} S^{-1}AS = D_A^{\fourdiag} = \left[ \tikz[baseline=+4mm]{ \draw[color=black] (0,1) -- (1,0); \draw[color=black] (0,0.5) -- (0.5,0); \draw[color=black] (0.5,1) -- (1,0.5); } \right] \label{equ:normal2-4diags} \end{equation}
is a matrix in four-diagonal-form.
\end{theorem}

\begin{proof}
Let $A \in \MM_{2m}(\CC)$ be $J_{2m}$-normal and diagonalizable and set $n=2m$. According to Lemma \ref{lem:perp2symp} the matrix $A' := UAU^H = UAU^{-1}$ for $U$ as defined in \eqref{equ:U-transform} is $R_{n}$-normal. As $A'$ is still diagonalizable, Theorem \ref{thm:normal2X} applies and there exists a perplectic matrix $P \in \Gl_{n}(\CC)$ such that
$$P^{-1}A'P = X_{A'} = \begin{bmatrix} D_{11} & R_nD_{12} \\ R_nD_{21} & D_{22} \end{bmatrix} = \left[ \tikz[baseline=+4mm]{ \draw[color=black] (0,1) -- (1,0); \draw[color=black] (0,0) -- (1,1); } \right]$$
is in $X$-form (where $D_{11}, D_{12}, D_{21}, D_{22} \in \MM_m(\CC)$ are diagonal). As $X_{A'}$ is $R_{n}$-normal (cf. Lemma \ref{lem:switch2}), the matrix $A'' := U^HX_{A'}U$ is $J_{2m}$-normal by Lemma \ref{lem:perp2symp} (c).  In conclusion, we have $S^{-1}AS = A''$ for the matrix $S := U^HPU$.
In accordance with Lemma \ref{lem:perp2symp} (d) the matrix $S \in \Gl_{2m}(\CC)$ is symplectic. Finally,
\begin{align*}
S^{-1}AS &= U^HX_{A'}U = \begin{bmatrix} I_m & \\ & iR_m \end{bmatrix} \begin{bmatrix} D_{11} & R_mD_{12} \\ R_mD_{21} & D_{22} \end{bmatrix} \begin{bmatrix} I_m & \\ & -iR_m \end{bmatrix} \\
&= \begin{bmatrix} D_{11} & -i R_m D_{12} R_m \\ iD_{21} & R_m D_{22} R_m \end{bmatrix} = \left[ \tikz[baseline=+4mm]{ \draw[color=black] (0,1) -- (1,0); \draw[color=black] (0,0.5) -- (0.5,0); \draw[color=black] (0.5,1) -- (1,0.5); } \right] =: D_A^{\fourdiag}
\end{align*}
is a matrix in four-diagonal-form and the statement is proven.
\end{proof}

Through the transformation with $U$ in \eqref{equ:U-transform}, the density result on $R_n$-normal matrices with $n$ distinct eigenvalues directly carries over to $\mathcal{N}(J_{2m})$.

\begin{theorem} \label{thm:approx-distinct-symp}
Let $A \in \MM_{2m}(\CC)$ be $J_{2m}$-normal and let $\varepsilon > 0$ be given. Then there exists some $J_{2m}$-normal matrix $\hat{A} \in \MM_{2m}(\CC)$ with $2m$ distinct eigenvalues such that $\Vert A - \hat{A} \Vert_2 < \varepsilon$.
\end{theorem}

\begin{proof}
Let $A \in \MM_{2m}(\CC)$ be $J_{2m}$-normal ($n=2m$) and let $\varepsilon > 0$ be given. Consider $A' := UAU^H$ for $U \in \Gl_{2m}(\CC)$ as in \eqref{equ:U-transform}. Then $A'$ is $R_{n}$-normal and, according to Theorem \ref{thm:approx-distinct}, there exists some $R_{n}$-normal matrix $\tilde{A} \in \MM_{n}(\CC)$ with $n$ distinct eigenvalues such that $\Vert A' - \tilde{A} \Vert_2 < \varepsilon$. Now define $\hat{A} := U^H \tilde{A} U$ which is $J_{2m}$-normal according to Lemma \ref{lem:perp2symp} (d). Then
$$\Vert A - \hat{A} \Vert_2 = \Vert U^HA'U  - U^H \tilde{A} U \Vert_2 = \Vert A' - \tilde{A} \Vert_2 < \epsilon$$
since $U$ is unitary and $\Vert \cdot \Vert_2$ is a unitarily invariant matrix norm. This completes the proof as the $n$ distinct eigenvalues of $\tilde{A}$ are exactly those of $\hat{A}$.
\end{proof}

Theorem \ref{thm:approx-distinct-symp} shows that the set of $J_{2m}$-normal matrices with $2m$ distinct eigenvalues is dense in $\mathcal{N}(J_{2m})$.
For the same reasoning as outlined subsequently to Theorem \ref{thm:approx-distinct} these matrices form an open and dense subset of $\mathcal{N}(J_{2m})$ with respect to the subspace topology on $\mathcal{N}(J_{2m})$. As all matrices in $\mathcal{N}(J_{2m})$ with $2m$ distinct eigenvalues are diagonalizable (and, according to \cite[Thm.\,9]{CruzSalten} matrices with $2m$ distinct eigenvalues are also dense in classes of skew-Hamiltonian, Hamiltonian and symplectic matrices), we obtain the analogous result to Theorem \ref{thm:genericity} for $J_{2m}$-normal matrices.

\begin{theorem} \label{thm:genericity-symp}
With $\mathcal{S} \subseteq \MM_{2m}(\CC)$ being one of the sets of all Hamiltonian, skew-Hamiltonian, symplectic or $J_{2m}$-normal matrices, the following is true:
the set of all matrices $A \in \mathcal{S}$ for which a symplectic matrix $S \in \Gl_{2m}(\CC)$ exists such that $S^{-1}AS = D_A^{\fourdiag}$ is in four-diagonal-form is open and dense. Consequently, the set of all matrices $A \in \mathcal{S}$ for which such a symplectic similarity transformation does not exist is nowhere dense in $\mathcal{S}$.
\end{theorem}

For an analogous reasoning as outlined in Section \ref{ss:genericity}, the set of $J_{2m}$-normal matrices for which a symplectic similarity transformation to four-diagonal-form does not exist can be considered as small from a topological point of view. For this reason, we call the four-diagonal form \eqref{equ:4diagonals} \emph{generic} for $J_{2m}$-normal matrices.

\section{Conclusion}
\label{s:conc}

We introduced a canonical form for matrices that are nondefective and normal with respect to the perplectic scalar product $[x,y] = x^HR_ny$ and the symplectic scalar product $[x,y] = x^HJ_{2m}y$. Such matrices need not be perplectically/symplectically diagonalizable in contrast to euclidean normal matrices (i.e. those matrices $A$ satisfying $A^HA = AA^H$) which are always unitarily diagonalizable. We showed that nondefective $R_n$-normal matrices can always be transformed into '$X$-form' via a perplectic similarity whereas diagonalizable $J_{2m}$-normal matrices are always symplectically similar to a matrix in 'four-diagonal-form'. According to \cite{CruzSalten} the diagonalizable matrices form a dense subset among all $R_n$/$J_{2m}$-normal matrices. We used the established canonical forms to prove that the set of all $R_n$/$J_{2m}$-normal matrices with pairwise distinct eigenvalues constitute an open and dense subset among all $R_n$/$J_{2m}$-normal matrices. From this it followed that the canonical forms can be seen as ``generic'' for these type of matrices since the set for which they do not exist is nowhere dense  and thus, from a topological point of view, small.

%\section*{Acknowledgement}




%\section*{References}
%\bibliographystyle{elsarticle-harv}
%\bibliography{mybib}


\end{document}